\documentclass[11pt]{amsart}

\usepackage{amsmath}
\usepackage{amssymb}
\usepackage{graphicx}

\hoffset = - 0.5in

\newtheorem{theorem}{Theorem}[section]
\newtheorem{corollary}[theorem]{Corollary}

\newtheorem{proposition}[theorem]{Proposition}
\theoremstyle {definition}

\numberwithin{equation}{section}

\renewcommand{\geq}{\geqslant}
\renewcommand{\leq}{\leqslant}

\title{Uniformly nonsquare Banach spaces have the fixed point property 1}

\author[Tim Dalby]{Tim Dalby}

\date{\today}



\email{tim\_dalby@bigpond.com}

\begin{document}
\parindent = 0pt
\parskip = 8pt

\begin{abstract}

Another proof that uniformly nonsquare Banach spaces have the fixed point property is presented.

\end{abstract}

\maketitle

\section{Introduction}

In [3] Garc\'{i}a-Falset, Llorens-Fuster and Mazcu\~{n}an-Navarroa were the first to show that a uniformly nonsquare Banach space has the Fixed Point Property, FPP.  The proof is not direct as it uses the following chain of reasoning.  The notation and definitions are explained in the next section.

A Banach space, $X$, is uniformly nonsquare if and only if  $\rho'_X(0) < 1$, see [6].  This then implies that $X$ is reflexive and  $\Gamma_X'(0) < 1.$   It is then shown that this implies $RW(a, X) < 1 + a$ for some $a > 0.$  Finally, it was shown that $R(a, X) \leq RW(a, X)$ which allowed a result of Dom\'{i}nguez-Benavides, [1], to be used to obtain the fixed point property.

In summary, the proof goes along the following path. The reverse implications are left out for ease of reading.

$X$ is uniformly nonsquare $\Rightarrow \rho'_X(0) < 1 \Rightarrow X \mbox{ is reflexive and } \Gamma_X'(0) < 1 \Rightarrow \mbox {there exist } a > 0 \mbox { such that } RW(a, X) < 1 + a \Rightarrow R(a, X) < 1 + a \Rightarrow X$ has the FPP.

The proof really shows the more general result that if $X$ is reflexive and $\Gamma_X'(0) < 1$ then $X$ has the FPP with the main result coming as a corollary.

The proof in this paper goes down a more direct path.  The idea is highlight the connection between the Fixed Point Property and the modulus of nearly uniform smoothness.

There will be a further two more papers on this topic presenting two more proofs that the property of uniformly nonsquare implies the FPP.  These proofs will be from slightly different perspectives with intention of highlighting different aspects of Banach space geometry.

At this point, it is important to note that there is now a very direct proof of uniformly nonsquare implying the FPP curtesy of Dowling, Randrianantoanina and Turett, [2].

Now for the definitions and notation.

\section{Definitions}

The modulus of uniform smoothness is 
\[\rho_X(t) = \sup \left\{ \frac{\displaystyle \| x + ty \| + \| x - ty \|}{\displaystyle 2} - 1: x, y \in B_X  \right\}\]
 where $t \geq 0$ and 
\[ \rho'_X(0) = \lim_{t \rightarrow 0^+} \frac{\rho_X(t)}{t}. \]

The modulus of nearly uniform smoothness is
\[ \Gamma_X(t) = \sup \left\{ \inf_{n > 1} \left( \frac{\| x_1 + tx_n \| + \| x_1 - tx_n \|}{2} - 1 \right) \right\}, \]
where the supremum is taken over all basic sequences $(x_n)$ in $B_X$.

If $X$ is reflexive then
\[ \Gamma_X(t) = \sup \left\{ \inf_{n>1} \left( \frac{\| x_1 + tx_n \| + \| x_1 - tx_n \|}{2} - 1 \right): (x_n) \mbox{ in } B_X, x_n \rightharpoonup 0 \right\}. \]
Also \[ \Gamma_X'(0) = \lim_{t \rightarrow 0 ^+} \frac{\Gamma_X(t)}{t}. \]

It is clear from the definitions that $\Gamma_X(t) \leq \rho_X(t) \mbox{ for all } t \geq 0.$

So $\Gamma_X'(0) \leq \rho'_X(0).$

In [3] it was shown that $\Gamma_X'(0) < 1 $ is equivalent to there exists $s > 0$ such that $\Gamma(s ) < s.$ \qquad  \dag

{\bf Remark} It can be shown that this last condition on $\Gamma(s)$ is equivalent to \newline $\Gamma(s) < s \mbox{ for all } s > 0.$

\section{Result}

\begin{proposition}

Let $X$ be a reflexive Banach space where $\Gamma_X'(0) < 1$ then $X$ has the FPP.

\end{proposition}

\begin{proof} 

Assume $X$ is a reflexive Banach space with $\Gamma_X'(0) < 1$.  From \dag \, we can find an $s > 0$ such that $\Gamma(s ) < s.$  

Let $t =\frac{\displaystyle s}{\displaystyle 1 + s}$ then $0 < t < 1.$

The usual set up is to assume $X$ does not have the FPP and arrive at a contradiction.  Using results from Goebel [4], Karlovitz [5] and Lin [7] plus an excursion to \newline $\tilde{X} = l_\infty(X)/c_0(X)$ and then back to $X$ it can be shown that there exists a sequence with the following properties.
\[ y_n \rightharpoonup y, \lim_{n \rightarrow \infty} \| y_n \|\mbox{ exists, } D[(y_n)] = \limsup_{n \rightarrow \infty} \limsup_{m \rightarrow \infty} \| y_n - y_m \| \leq t  \mbox{ and } \| y \| \leq 1 - t. \]

To obtain the required contradiction, we need to show that $\lim_{n \rightarrow \infty} \| y_n \|$ is uniformly away from 1.

Let $w_n = y_n - y \mbox{ and } w = y \mbox{ then } w_n \rightharpoonup 0.$

Since $y_n = w_n + w$ it is necessary to investigate upper bounds on $\| w_n + w \|.$  This is done via $D[(w_n)] = D[(y_n)] \leq t.$

It turns out that $\lim_{n \rightarrow\infty}\| w_n + w \|$ bounded above by both 
\[ \liminf_{n \rightarrow \infty}\liminf_{m \rightarrow \infty} \| (w_n - w_m) + w \| \mbox{ and } \liminf_{n \rightarrow \infty}\liminf_{m \rightarrow \infty} \| (w_n - w_m) - w \| \]
as shown below.

Using lemma 3.2 in [3] there exists a subsequence $(w_{n_k})$ such that
\[ \limsup_{k \rightarrow \infty} \| w_{n_k} - w_{n_{k + 1}} \| \leq D[(w_n)] = D[(y_n)] \leq t \]
and
\[ \liminf_{k \rightarrow \infty} \| w_{n_k} - w_{n_{k + 1}}+w \| \geq \liminf_{n \rightarrow \infty}\liminf_{m \rightarrow \infty} \| (w_n - w_m) + w \|  \]
and 
\[ \liminf_{k \rightarrow \infty} \| w_{n_k} - w_{n_{k + 1}} - w \| \geq \liminf_{n \rightarrow \infty}\liminf_{m \rightarrow \infty} \| (w_n - w_m) - w \|. \]

Using the weak lower semicontinuity of the norm, 
\[ \liminf_{m \rightarrow \infty} \| (w_n - w_m) + w \| \geq \| w_n + w \| \mbox{ for all } n. \]

So $\liminf_{n \rightarrow \infty} \liminf_{m \rightarrow \infty} \|( w_n - w_m) + w \| \geq \liminf_{n \rightarrow \infty} \| w_n + w \|.$

To obtain a similar inequality for $\liminf_{n \rightarrow \infty}\liminf_{m \rightarrow \infty} \| (w_n - w_m) - w \|$ take $w_n^* \in S_{X^*} \mbox{ where } 
w^*_n(w_n + w) = \| w_n + w \|.$  Because of the reflexivity of $X, B_{X^*}$ is w*-sequentially compact.  So, without loss of generality, we may assume that $w_n^* \stackrel{*}{\rightharpoonup} w^* \mbox { where } \| w^* \| \leq 1.$

Now
\begin{align*}
\liminf_{m \rightarrow \infty} \| (w_n - w_m) - w \| & \geq \liminf_{m \rightarrow \infty} (-w_m^*)((w_n - w_m) - w)\\
& = \liminf_{m \rightarrow \infty} w_m^*(w_m + w) - w^*(w_n)\\
& = \liminf_{m \rightarrow \infty} \|w_m + w\| - w^*(w_n).
\end{align*}

Therefore $\liminf_{n \rightarrow \infty} \liminf_{m \rightarrow \infty} \| (w_n - w_m) - w \| \geq \liminf_{m \rightarrow \infty} \| w_m + n \|.$

Thus 
\begin{align*}
\lim_{n \rightarrow\infty} \| y_n \| & = \lim_{n \rightarrow \infty} \| w_n + w \|\\
& \leq \left ( \liminf_{k \rightarrow\infty} \| (w_{n_k} - w_{n_{k + 1}}) + w \| \right ) \wedge \left (\liminf_{k \rightarrow\infty} \| (w_{n_k} - w_{n_{k + 1}}) - w \| \right )\\
& \leq \frac{\liminf_{k \rightarrow\infty} \| (w_{n_k} - w_{n_{k + 1}}) + w \| + \liminf_{k \rightarrow \infty} \| (w_{n_k} - w_{n_{k + 1}}) - w \|}{2}.
\end{align*}

This expression is starting to look like that used in the definition of $\Gamma(s)$ since $(w_{n_k} - w_{n_{k + 1}})$ is weak null.

Next is some scaling so that $(w_{n_k} - w_{n_{k + 1}})$ becomes a weak null sequence in $B_X$ and $w$ is scaled so it too is in $B_X.$

Choose $\alpha > 0$ such that  $\alpha < \frac{\displaystyle s (s - \Gamma(s) )}{ \displaystyle ( \Gamma(s) + s + 2 )( 1 + s )}. \qquad \dag\dag$

This unusual fraction is chosen so that three inequalities a little further on, work out just right.

This means
\begin{align*}
\alpha(1 + s) & < \frac{s (s - \Gamma(s) )}{\Gamma(s) + s + 2}\\
& = \frac{2s (1 + s)}{\Gamma(s) + s + 2} - s.
\end{align*}

Which leads to
\[ s + \alpha(1 + s) < \frac{2s (1 + s)}{\Gamma(s) + s + 2} \]

and
\[ \frac{s + \alpha(1 + s)}{s (1 + s)} \left ( \frac{\Gamma(s) + s + 2}{2} \right ) < 1. \]

Let 
\[ v_1 = \frac{ \displaystyle w}{ \displaystyle \frac{ \displaystyle s + \alpha(1 + s)}{ \displaystyle s (1 + s)}} \]

and
\[v_k = \frac{ \displaystyle w_{n_k} - w_{n_{k + 1}}}{ \displaystyle \frac{ \displaystyle s + \alpha(1 + s)}{ \displaystyle 1 + s}} \mbox{ for } k > 1. \]

Then $v_k \rightharpoonup 0 \mbox{ and for large enough } k,  \| w_{n_k} - w_{n_{k + 1}} \| \leq t + \alpha$.

So for large $k$,

\begin{align*}
\| v_k \|  & \leq \frac{ \displaystyle t + \alpha}{ \displaystyle \frac{s + \alpha(1 + s)}{ \displaystyle 1 + s}}\\
& = \frac{s + \alpha(1 + s)}{s + \alpha(1 + s)}\\
& = 1.
\end{align*}
Also 
\begin{align*} 
\| v_1 \| & \leq \frac{ \displaystyle 1 - t}{ \displaystyle \frac{s + \alpha(1 + s)}{ \displaystyle s (1 + s)}}\\
& = \frac{s (1 + s)}{s + \alpha(1 + s)}\left ( 1 - \frac{s}{1 + s} \right )\\
& = \frac{s}{s + \alpha(1 + s)}\\
& \leq 1. 
\end{align*}

Now back to the story about $\lim_{n \rightarrow\infty}\| w_n + w \|.$

\[ \frac{ \displaystyle 1}{ \displaystyle \frac{s + \alpha(1 + s)}{ \displaystyle s (1 + s)}}\lim_{n \rightarrow\infty} \| w_n + w \| \leq \frac{\liminf_{k \rightarrow\infty}\| sv_k + v_1 \| + \liminf_{k \rightarrow\infty}\| sv_k - v_1 \|}{2}. \]

But
\[ \inf_{k > 1}\left ( \frac{\displaystyle \| sv_k + v_1 \| + \|sv_k - v_1 \|}{\displaystyle 2} - 1 \right ) \leq \Gamma(s). \]

So there exists $k_0$ such that
\[ \frac{\displaystyle \| sv_{k_0} + v_1 \| + \| sv_{k_0} - v_1 \|}{\displaystyle 2} - 1 \leq \Gamma(s) + \frac{\displaystyle s - \Gamma(s)}{\displaystyle 2} = \frac{\displaystyle \Gamma(s) + s}{\displaystyle 2}. \]

If the little bit added onto $\Gamma(s), \frac{\displaystyle s - \Gamma(s)}{\displaystyle 2}$, is different then so will be the expression in inequality \dag\dag.

Now
\begin{align*}
\frac{\displaystyle \liminf_{k \rightarrow\infty} \| sv_k + v_1 \| + \liminf_{k \rightarrow\infty} \| sv_k - v_1 \|}{\displaystyle 2} - 1 & \leq \frac{\displaystyle \| sv_{k_0} + v_1 \| + \| sv_{k_0} - v_1 \|}{\displaystyle 2} - 1\\
& \leq \frac{\displaystyle \Gamma(s) + s}{\displaystyle 2}.
\end{align*}

Therefore

\begin{align*}
\lim_{n \rightarrow\infty}\| y_n \| & = \liminf_{n \rightarrow\infty} \| w_n + w \|\\
& \leq \frac{s + \alpha(1 + s)}{s(1 + s)} \times \frac{\Gamma(s) + s + 2}{2}\\
& < 1.
\end{align*}

A contradiction, so $X$ has the FPP.

\end{proof}

\begin{corollary}

If $X$ is a Banach space that is uniformly nonsquare then $X$ has the FPP.

\end{corollary}

{\bf Remark}

 \begin{enumerate}

\item [1.]   If $X$ is uniformly nonsquare then there exists $s > 0$ such that $\rho(s)  < s.$  Can a proof be devised where this inequality is used instead of there exists $s > 0$ such that $\Gamma(s) < s?$

\item [2.] Following on from the Remark in Section 2, since any $s > 0$ can be used, does this allow a similar result but where the Banach-Mazur distance is used.  That is, if $X$ is a uniformly nonsquare Banach space and $Y$ is another Banach space where the Banach-Mazur distance, $d(X,Y) < \beta$ then $Y$ has the FPP.  Here the $\beta$ is yet to be determined.

\end{enumerate}


\begin{thebibliography}{99}

\bibitem {1} T. Dom\'{i}nguez-Benavides, {\it A geometric coefficient implying the fixed point property and stability results}, Houston J. Math. {\bf 22} (1996), 835-849.

\bibitem {2} P. N. Dowling, B. Randrianantoanina and B. Turett, {\it The fixed point property via dual space properties}, J. Funct. Anal. {\bf 255} (2008), 768-775.

\bibitem {3} J. Garc\'{i}a-Falset, E. Llorens-Fuster and E. M. Mazcu\~{n}an-Navarroa, {\it Uniformly nonsquare Banach spaces have the fixed point property for nonexpansive mappings}, J. Funct. Anal. {\bf 233} (2006), 494-534.

\bibitem {4} K. Goebel, {\it On the structure of minimal invariant sets for nonexpansive mappings}, Ann. Univ. Mariae Curie-Sk\l odowska Sect. A {\bf 29} (1975), 73-77.

\bibitem {5} L. A. Karlovitz, {\it Existence of fixed points for nonexpansive map in a space without normal structure}, Pacific J. Math. {\bf 66} (1976), 153-159.

\bibitem {6}M. Kato, L. Maligranda and Y. Takahashi, {\it On James and Jordan-von Neumann constants and the normal structure coefficient of Banach spaces}, Studia Math. {\bf 144} (2001), 275-295.

\bibitem {7} P.-K. Lin, {\it Unconditional bases and fixed points of nonexpansive mappings}, Pacific J. Math. {\bf 116} (1985), 69-76.

\end{thebibliography}
\end{document}